\documentclass{amsart}

\usepackage{a4,amsmath,amssymb,amscd,verbatim}

\newtheorem{proposition}[equation]{Proposition}
\newtheorem{theorem}[equation]{Theorem}
\newtheorem{corollary}[equation]{Corollary}
\newtheorem{lemma}[equation]{Lemma}

\theoremstyle{definition}

\newtheorem{definition}[equation]{Definition}

\newtheorem{remark}[equation]{Remark}

\numberwithin{equation}{section}

\newcommand{\srf}[1]{\mbox{${\text{\it SR}^F}$}}

\def\geq{\geqslant}

\def\leq{\leqslant}

\begin{document}
\bibliographystyle{plain} \title[Spectral triples and Gibbs measures]% \lebn]
{Spectral triples and Gibbs measures for expanding maps on Cantor sets} 

\author{Richard Sharp} \address{School of Mathematics,
  University of Manchester, Oxford Road, Manchester M13 9PL, England}
\email{sharp@maths.man.ac.uk}
%\author{Taras E Panov}
%\address{Department of Mathematics and Mechanics, Moscow State
%University, Leninskie Gory, 119992 Moscow, Russia; \emph{and}
%Institute for Theoretical and Experimental Physics, Moscow 117259, Russia}
%\email{tpanov@mech.math.msu.ru}

%\thanks{}


%\date{\today}

\begin{abstract}
  Let $T : \Lambda \to \Lambda$ be an expanding map on a Cantor set.
For each suitably normalized H\"older continuous potential, we construct a spectral triple from which
one may recover the associated Gibbs measure as a noncommutative measure.
\end{abstract}

\maketitle

%
%
%
%
%
%
%
%
%
%%%%%%%%%%%%%%%%%%%%%%%%%%%%%%%%%%%%%%%%%%%%%%%%%%%%%%%%%%%%%%%%%%%%%%%%%%%%%

\section{Introduction}\label{in}

%\begin{equation}\label{prodsphs}
%\varPi\?S(q)\;\letbe\;\prod^kS^{2q+1}\subset\mathbb{C}^{k(q+1)}
%\end{equation}

%\[
%  t\cdot(x,y)\;=\;(t^{-1}x,ty)
%\]

In the 1990s, Connes introduced the concept of a {\it spectral triple}
as a fundamental object in noncommutative geometry \cite{Co1,Co2,Co3, Var},
giving a ``space-free'' description of many geometric phenonmena.
The notion is very flexible and appropriate choices allow one to
recover
the volume measure and metric on a Riemannian spin
manifold \cite{Co1,Co2} and also,
for example, the Hausdorff measure on certain fractal sets \cite{Co1,
GI1, GI2, GI3}. 
In the fractal case, 
the starting point is Connes's construction of a spectral triple for a
Cantor subset of the real line, from which (using ideas of Lapidus
and Pomerance \cite{LP}) the Minkowski content may be recovered.
For simple self-similar sets,  
the Hausdorff measure may
also be obtained and Guido and Isola generalize these ideas to certain
fractal subsets
of $\mathbb R^n$.
(See \cite{CI, CS} for other approaches and \cite{Palm} for a more
general construction valid for any compact metric space.) Furthermore,
Falconer and Samuel \cite{FS} have modified this construction
to describe multifractal phenomena.

The purpose of this paper is to show that, for a class of expanding
maps, certain important measures,
called Gibbs measures,
which arise in the ergodic theory of 
hyperbolic dynamical systems, may be obtained as noncommutative
measures
from an appropriate spectral triple. 
Specifically, our dynamical systems will be
expanding maps conjugate to a subshift of finite type (not necessarily
a full shift), so that,
in particular, the (maximal) invariant set is a
Cantor set. After this paper was written, we learned that Samuel
had obtained a very similar result in his thesis \cite{Sam1}. We will
give a little more detail on his work following Theorem
\ref{noncommmeas}
below.

We shall we shall now fix some notation.
Let $\Lambda$ be a compact subset of a smooth Riemannian manifold
$M$ and let $T : \Lambda \to \Lambda$ be a  $C^1$ 
expanding map which is  topologically conjugate to 
a mixing one-sided subshift of finite type $\sigma : \Sigma_A^+ \to
\Sigma_A^+$.
(See section 3 below for precise definitions.)
In particular, $\Lambda$ is a Cantor set. 
The purpose of this
condition
is to ensure that so-called ``locally constant'' functions on 
$\Lambda$ are contained
in $C(\Lambda,\mathbb C)$.

Let $\mathcal M_T$ denote the set of $T$-invariant probability
measures on $\Lambda$. This is a large set
but we may  single out the so-called {\it Gibbs measures}
(or equilibrium measures) 
associated to H\"older continuous potentials as being of particular
importance. 
These are defined as follows. Let $\psi : \Lambda \to \Lambda$ be
a H\"older continuous function. Then the Gibbs measure for $\psi$ is the
unique $\mu \in  \mathcal M_T$ for which
\begin{align*}
h_T(\mu) + \int \psi \, d\mu = \sup_{m \in \mathcal M_T}
\left(h_T(m) + \int \psi \, dm\right),
\end{align*}
where $h_T(m)$ denotes the entropy of $T$ with respect to $m$.
A general feature of hyperbolic dynamical systems is that averages
of weighted local local data (e.g. sums of observables over sets of
orbits)
give global information (e.g. the average of an observable with
respect to an invariant measure) \cite{Ba1, PP, Sh1}
and Gibbs measures may be obtained in this way. 
(Very roughly, weighting by the exponentials of sums of an observable 
$\psi$ gives the Gibbs measure
for $\psi$.) However, this
local
to global property also motivates the definition of a Dirac operator,
adapted from those in \cite{FS, GI1, GI2, GI3},
and allows us to obtain a noncommutative integral from its spectrum.
In Theorem \ref{noncommmeas} below, we show that this noncommutative
integral
agrees up to an explicit factor with the integral with respect
the Gibbs measure.
We begin by defining a spectral triple \cite{Co1, Var}.

\begin{definition} A {\it spectral triple} is a triple
$(H,A,D)$, where 
\begin{itemize}
\item[(i)]
$H$ is a Hilbert space;
\item[(ii)]
$A$ is a $C^*$-algebra equipped with a faithful
representation $\pi : A \to B(H)$ (the bounded linear operators
on $H$);
\item[(iii)]
$D$ 
is an essentially self-adjoint unbounded linear operator on $H$ 
with compact resolvent and such that
$\{f \in A \hbox{ : } \|[D,\pi(f)]\| < +\infty\}$ is
dense in $A$, where $[D,\pi(f)] : H \to H$ is the commutator operator
$[D,\pi(f)](\xi) = D\pi(f)(\xi) - \pi(f)D(\xi)$.
This $D$ is called a Dirac operator.
\end{itemize}
\end{definition}

We shall define spectral triples associated to H\"older
continuous potentials on
$\Lambda$, adapting the constructions of Connes \cite{Co1},
Guido and Isola \cite{GI1, GI2, GI3}) and Falconer and Samuel \cite{FS}.
As above, let $\sigma : \Sigma_A^+ \to \Sigma_A^+$ be the subshift of
finite type topologically conjugate to the expanding map
$T : \Lambda \to \Lambda$,
where the symbol set is $\{1,\ldots,k\}$
and $A$ is a zero-one transition matrix.
(See section 3 for a complete definition.)
We shall write $p : \Sigma_A^+ \to \Lambda$ for the conjugating 
homeomorphism. In the interests of readability, we will
systematically abuse notation by writing
$f(x)$ instead of $f(p(x))$, whenever $f \in C(\Lambda,\mathbb C)$ and
$x \in \Sigma_A^+$.

An ordered $n$-tuple
$(w_1,\ldots,w_n)$, with $w_m \in \{1,\ldots,k\}$, $m=1,\ldots,n$,
is called an {\it allowed word} of length $n$ 
if $A(w_m,w_{m+1})=1$ for $m=1,\ldots, n-1$.
Let $W_n$ denote the set of allowed words of length $n$ and let
\begin{align*}
W^* = \bigcup_{n=1}^\infty W_n.
\end{align*}
For $w = (w_1,\ldots,w_n) \in W_n$, we write
\begin{align*}
[w] = \{x = (x_n)_{n=1}^\infty \in \Sigma_A^+ \hbox{ : } x_m =w_m, \
m=1,\ldots,n\}
\end{align*} 
and 
$\mathfrak t(w) = w_n$.
For $w \in W_n$ and $x \in \Sigma_A^+$ then $wx$ will denote the
sequence defined by
\begin{align*}
(wx)_m = \begin{cases}
w_m &\hbox{ if }  1 \leq m \leq n \\
x_{m-n}
&\hbox{  if } m \geq n+1. \end{cases}
\end{align*}
Clearly, $wx \in \Sigma_A^+$ if and only if $A(\mathfrak t(w),x_1)=1$.
For each $j \in \{1,\ldots,k\}$, choose a sequence $x^{j} \in
\Sigma_A^+$
and distinct sequences
$y^{j},z^{j} \in \Sigma_A^+$ such that
$jx^{j},jy^{j},jz^{j} \in \Sigma_A^+$.

Now we can define a spectral triple associated to a continuous
potential $\phi : \Lambda \to \mathbb R$.
Our Hilbert space will be
\begin{align*}
H = \ell^2(W^*) \oplus \ell^2(W^*) \subset \bigoplus_{w \in W^*}
\mathbb C \oplus \mathbb C,
\end{align*}
where we write a typical element as
\begin{align*}
\xi =\bigoplus_{w \in W^*} \left(\begin{matrix} \xi_1(w) \\
    \xi_2(w) \end{matrix} \right)
\end{align*}
and our $C^*$-algebra
will be
$A = C(\Lambda,\mathbb C)$.
We define a $*$-representation $\pi : A \to B(H)$ by setting
$\pi(f)$ to be the multiplication operator
$$\pi(f)\left(\bigoplus_{w \in W^*}
\left(\begin{matrix} \xi_1(w) \\ \xi_2(w) \end{matrix} \right)\right)
=\bigoplus_{w \in W^*}
\left(\begin{matrix} f\left(wy^{\mathfrak t(w)}\right) \xi_1(w) \\ 
f\left(wz^{\mathfrak t(w)}\right)\xi_2(w) \end{matrix} \right),$$
We define $D_\phi : H \to H$ by
\begin{align*}
D_{\phi}\left(\bigoplus_{w \in W^*} \left(\begin{matrix} \xi_1(w) \\ \xi_2(w) \end{matrix} \right)\right)
&= \bigoplus_{n=1}^\infty\bigoplus_{w \in W_n}
e^{\phi^n(wx^{\mathfrak t(w)})} \left(\begin{matrix} 0 & 1 \\ 1 & 0 \end{matrix} \right)
\left(\begin{matrix} \xi_1(w) \\ \xi_2(w) \end{matrix} \right) \\
&= \bigoplus_{n=1}^\infty \bigoplus_{w \in W_n}
e^{\phi^n(wx^{\mathfrak t(w)})} \left(\begin{matrix} \xi_2(w) \\ \xi_1(w) \end{matrix}
\right),
\end{align*}
where
$\phi^n := \phi + \phi \circ T + \cdots + \phi \circ T^{n-1}$.
We have the following theorem.

\begin{theorem} \label{spectraltriple}
For any continuous function $\phi : \Lambda \to \mathbb R$,
$(H,A,D_\phi)$ is a spectral triple.
\end{theorem}

The main result of the paper is that, when $\phi$ is H\"older
continuous and is suitably normalized, 
we may recover the Gibbs measure for $-\phi$ from the operators
$\pi(f)|D_\phi|^{-1}$ via a singular trace.
(The choice of sign is for notational convenience.)
In the next section, we introduce the ideas needed to explain 
this statement and then state our main theorem.
In section 3, we discuss some material on expanding maps,
subshifts of finite type and transfer operators.
In section 4, we prove Theorem \ref{spectraltriple}.
In section 5, we complete the paper by proving our result on
noncommutative
measures and Gibbs measures, Theorem \ref{noncommmeas}.

\section{Singular traces and noncommutative measures}

In order to state our main result, 
we need to briefly discuss the theory of singular traces
of compact operators. For more details, see \cite{AGPS} or \cite{GI2}.
Let $B(H)$ denote the algebra of bounded linear operators on a Hilbert space 
$H$ and let $K(H)$ denote the ideal of compact operators.
A {\it singular trace} on a two-sided ideal $I \subset K(H)$ is a 
positive linear functional $\tau : I \to \mathbb R$ such that
$\tau$ is unitary invariant (the trace property) and
vanishes on finite rank operators.

The most important singular traces are the so-called {\it Dixmier
traces} \cite{Di}. These are defined on an ideal $I = \mathcal
L^{1,\infty}(H)$,
the {\it Dixmier ideal}, given by
\begin{align*}
\mathcal L^{1,\infty}(H) = \left\{A \in K(H) \hbox{ : }
\limsup_{n \to +\infty} \frac{1}{\log n} \sum_{k=1}^n a_k
<+\infty\right\},
\end{align*}
where $\{a_n\}_{n=1}^\infty$  denote the eigenvalues of
$|A| := \sqrt{A^*A}$, written in decreasing order.
Then
 a Dixmier trace is a singular trace $\tau_\omega$ on $\mathcal
L^{1,\infty}(H)$ defined, for a positive operator 
$A$, by
\begin{align*}
\tau_\omega(A) = \omega\text{-}\lim 
\frac{1}{\log n} \sum_{k=1}^n a_k,
\end{align*}
where this is a generalized limit corresponding to
a state $\omega$ on $l^\infty$, and extended to $\mathcal
L^{1,\infty}(H)$
by linearity. 
If the limit 
\begin{align*}
\lim_{n \to +\infty} 
\frac{1}{\log n} \sum_{k=1}^n a_k
\end{align*}
 exists then we say that $A$
is {\it measurable} and call the value of the limit the {\it noncommutative
  integral} of $A$.
(There are more general definitions
of the Dixmier trace 
--
see, for example, Chapter IV, \S2.$\beta$ of \cite{Co1},
\cite{LSS} or Chapter 5 of \cite{Var}. Correspondingly, there are more
general definitions of  measurability. It is shown in \cite{LSS} that
these are equivalent to the definition given here.)

Consider the spectral triple $(H,A,D_\phi)$ defined in the
previous section.
We will now suppose that $\phi : \Lambda \to \mathbb R$ is H\"older
continuous. 
We say that
$-\phi$ is {\it normalized} if 
\begin{align*}
\sum_{Ty=x} e^{-\phi(y)} =1,
\end{align*}
for all $x \in \Lambda$. (As we shall see in section 3,
any real-valued H\"older continuous function may be
normalized by adding a constant and a function of the form
$u \circ T -u$, with $u \in C(\Lambda,\mathbb R)$, and this operation
does not change
the Gibbs measure.)
Then, for $f \in C(\Lambda,\mathbb C)$, the operator
$\pi(f) |D_\phi|^{-1}$ is given by the formula
\begin{align*}
\pi(f)|D_\phi|^{-1}\left(\bigoplus_{w \in W^*} 
\left(\begin{matrix} \xi_1(w) \\ \xi_2(w) \end{matrix} \right)\right)
= \bigoplus_{n=1}^\infty \bigoplus_{w \in W_n} e^{-\phi^n(wx^{\mathfrak t(w)})}
\left(\begin{matrix} f\left(wy^{\mathfrak t(w)}\right) \xi_1(w) \\ 
f\left(wz^{\mathfrak t(w)}\right)\xi_2(w) \end{matrix} \right).
\end{align*}

\begin{theorem} \label{noncommmeas}
Suppose that $\phi : \Lambda \to \mathbb R$ is a H\"older continuous 
function and that $-\phi$ is normalized. Then, for any $f \in
C(\Lambda,\mathbb C)$,
\begin{itemize}
\item [(i)]
$\pi(f) |D_\phi|^{-1} \in \mathcal L^{1,\infty}(H)$; 
\item [(ii)] 
$\pi(f) |D_\phi|^{-1}$ is measurable and
\begin{align*}
\tau_\omega(\pi(f) |D_\phi|^{-1}) = c_\phi \int f \, d\mu,
\end{align*}
where $\mu$ is the Gibbs measure for $-\phi$ and where
\begin{align*}
c_\phi
= \frac{2}{\int \phi \, d\mu} \sum_{j=1}^k \sum_{Tx =x^j}
e^{-\phi(x)}
\chi_j(x),
\end{align*}
with $\chi_j$ the indicator function of the set $p([j])$.
\end{itemize}
\end{theorem}

\begin{remark}
As we noted in the introduction, a result similar to Theorem
\ref{noncommmeas}
has been obtained by Samuel \cite{Sam1}. A significant difference is
that he requires the potential $\phi$ to be non-arithmetic, i.e.,
that the sums of $\phi$ around periodic orbits do not all lie in a
single discrete
subgroup of $\mathbb R$. This restriction is needed for the
renewal
theory approach he uses. Thus, for example, his results do not cover
the measure of maximal entropy. A particularly attractive feature of
his work is that he is able to explicity calculate the
Dixmier trace associated to the constant function $1$, i.e.
$\tau_\omega(|D_\phi|^{-1})$, is equal to
 the reciprocal of the entropy of $\mu$ and he identifies this as a
 noncommutative volume.
\end{remark}

\section{Expanding maps and subshifts of finite type}\label{expand}

%\subsection{Background and examples}

We begin the section by defining subshifts of finite type.
Let $A$ be a $k \times k$ matrix whose
entries are all either zero or one. We define the (one-sided)
shift space
$$\Sigma_A^+ = \left\{(x_n)_{n=1}^\infty \in \prod_{n=1}^\infty
\{1,\ldots,k\}
\hbox{ : } A(x_n,x_{n+1})=1 \ \text{ for all } n \geq 1\right\}$$
and the (one-sided) subshift of finite type
$\sigma : \Sigma_A^+ \to \Sigma_A^+$
by $(\sigma x)_n =x_{n+1}$.
We give $\{1,\ldots,k\}$ the discrete topology, 
$\prod_{n=1}^\infty\{1,\ldots,k\}$
the product topology
and $\Sigma_A^+$ the subspace topology. A compatible metric
is given by
$$d((x_n)_{n=1}^\infty,(y_n)_{n=1}^\infty)
= \sum_{n=1}^\infty \frac{1-\delta_{x_n y_n}}{2^n},$$
where $\delta_{ij}$ is the Kronecker symbol.

We say that the matrix $A$ is irreducible if, for each
$(i,j)$, there exists $n(i,j) \geq 1$ such that
$A^{n(i,j)}(i,j)>0$ and aperiodic if there exists $n \geq 1$ such that,
for each $(i,j)$, $A^n(i,j)>0$.
The latter statement is equivalent to $\sigma : \Sigma_A^+ \to \Sigma_A^+$
being topologically mixing (i.e. that there exists $n \geq 1$
such that for any two non-empty open
sets $U,V \subset \Sigma_A^+$, $\sigma^{-m}(U) \cap V \neq \varnothing$,
for all $m \geq n$).

Let $M$ be a compact connected smooth Riemannian manifold and suppose that
$\Lambda \subset U \subset M$ with $\Lambda$ compact and $U$ open. 
Let $T : U \to M$ be a $C^1$ map. Suppose
that 
\begin{enumerate}
\item[(i)]
there exists $\lambda > 1$ such that $\|DT_x\| \geq
\lambda$ 
for all $x \in U$;
\item[(ii)]
$\Lambda = \bigcap_{n=0}^\infty T^{-n}U$;
\item[(iii)]
$T$ is topologically mixing.
\end{enumerate}
If $T$ satisfies (i), (ii) and (iii) then we refer to 
$T : \Lambda \to \Lambda$ as an expanding map and we
can find a mixing one-sided subshift of finite type $\sigma :
\Sigma_A^+ \to \Sigma_A^+$ and a H\"older continuous map
$p : \Sigma_A^+
\to \Lambda$ which semi-conjugates $T$ and $\sigma$. 
Furthermore, the map is ``nearly'' a homeomorphism. Here, however, we
impose the additional condition that $p$ is, in fact, a homeomorphism
and assume that
\begin{enumerate}
\item[(iv)]
$T : \Lambda \to \Lambda$ is topologically conjugate to
a mixing one-sided subshift of finite type $\sigma : \Sigma_A^+ \to
\Sigma_A^+$.
\end{enumerate}
In particular, (iv) implies that $\Lambda$ is a Cantor set.

Assumption (iv) gives $\Lambda$ a natural grading. In particular,
for each $n \geq 1$, we may write $\Lambda$ as a disjoint union
\begin{align*}
\Lambda = \bigcup_{w \in W_n} p([w]).
\end{align*}
We will say that a function  $f : \Lambda \to \mathbb C$ is 
{\it locally constant} if, for some $n \geq 1$, $f$ is constant on
each set $p([w])$, $w \in W_n$. We shall write
$\mathrm{LC}(\Lambda)$ for the set of all locally constant functions
on $\Lambda$. Clearly, $\mathrm{LC}(\Lambda)$ a uniformly dense
subalgebra of 
$C(\Lambda,\mathbb C)$.

We shall also consider some larger subalgebras of $C(\Lambda,\mathbb C)$.
For $\alpha >0$, we shall let $C^\alpha(\Lambda,\mathbb C)$
denote the space of $\alpha$-H\"older continuous functions on
$\Lambda$,
i.e., the set of functions $g : \Lambda \to \mathbb C$ satisfying
\begin{align*}
|g|_\alpha := \sup_{x \neq y} \frac{|g(x)-g(y)|}{d(x,y)^\alpha}
<+\infty.
\end{align*}
This is a Banach space with respect to the norm
$\|\cdot\|_\alpha = \|\cdot\|_\infty + |\cdot|_\alpha$.
Clearly, for any $\alpha>0$,
\begin{align*}
\mathrm{LC}(\Lambda) \subset C^\alpha(\Lambda,\mathbb C) \subset 
C(\Lambda,\mathbb C).
\end{align*}

\section{Gibbs states and transfer operators}\label{gibbs}

In this section we shall discuss some of the ergodic theory associate
to the map $T : \Lambda \to \Lambda$. 
The main references are \cite{Bo1} and \cite{PP}, where this theory is
developed for subshifts of finite type. The symbolic dynamics
described in the
preceding section allows the results to be immediately transfered
to expanding maps.
As above, we shall write
$\mathcal M_T$ for the space of $T$-invariant probability measures.
Given $m \in \mathcal M_T$, we write 
$h_T(m) \geq 0$ for the entropy of $T$ as a measure preserving
transformation
of $(\Lambda,m)$ (see \cite{Wa1} for the definition).
For a continuous function $\psi : \Lambda \to \mathbb R$, we define
its {\it pressure} $P(\psi)$ by
\begin{align*}
P(\psi) = 
\sup_{m \in \mathcal M_T}
\left(h_T(m) + \int \psi  \, dm\right).
\end{align*}
If $\psi$ is H\"older continuous, then there is a unique probability
measure $\mu$, called the {\it Gibbs measure} (or equilibrium measure)
for $\psi$, for which this
supremum is realized \cite{Ba1, Bo1, PP}.

Given $\psi \in C(\Lambda,\mathbb R)$, we define the {\it
transfer operator} $L_{\psi} : C(\Lambda,\mathbb C) \to
C(\Lambda,\mathbb C)$
by
\begin{align*}
L_\psi g(x) = \sum_{Ty=x} e^{\psi(y)} g(y).
\end{align*}
A key element of our approach will be to relate the eigenvalue
asymptotics of our Dirac operators to the spectral properties of
transfer operators. For this approach to work, we shall need 
to find a space of which $L_{\psi}$ acts quasi-compactly.

If $\psi \in C^\alpha(\Lambda,\mathbb R)$ then
$L_{\psi} : C^\alpha(\Lambda,\mathbb C) \to
C^\alpha(\Lambda,\mathbb C)$. The basic spectral properties of 
$L_{\psi}$ on this space are contained in the following result,
which is Ruelle's generalization of the classical Perron-Frobenius
Theorem for non-negative matrices).

\begin{proposition} \label{rpf} \cite{Ba1, Bo1, PP, Ru1}
If $\psi \in C^\alpha(\Lambda,\mathbb R)$ then 
$L_\psi : C^\alpha(\Lambda,\mathbb C) \to
C^\alpha(\Lambda,\mathbb C)$ has a simple eigenvalue equal to
$e^{P(\psi)}$ with the rest of the spectrum contained
in a disk $\{z \in \mathbb C \hbox{ : } |z| \leq \theta e^{P(\psi)}\}$,
for some $0 < \theta<1$. Furthermore, there exist
\begin{itemize}
\item[(i)] a
strictly positive eigenfunction $h \in C^\alpha(\Lambda,\mathbb R)$
such
that $L_\psi h=e^{P(\psi)}h$; and
\item[(ii)]
an eigenmeasure $\nu \in C(\Lambda,\mathbb R)^*$ such that 
$L_\psi^*\nu = e^{P(\psi)}\nu$.
\end{itemize} 
If $\nu$ is chosen to be a probability measure and
the eigenfunction 
$h$ is chosen so that $\int h \, d\nu =1$ then $\mu = h \nu$ is the 
Gibbs measure for $\psi$.
\end{proposition}

\begin{corollary} \label{spectraldecomp} There exists $\lambda_\psi <
  e^{P(\psi)}$ 
such
  that,
for any $f \in
  C^\alpha(\Lambda,\mathbb C)$,
we have
\begin{align*}
L_\psi^n f = \left(\int f \, d\nu\right) h e^{nP(\psi)}+ O(\lambda_\psi^n).
\end{align*} 
\end{corollary}

\begin{proof}
We recall the following basic fact from spectral theory 
(see, for example, \cite{BP} or \cite{Kato}).
Let $L : B \to B$ be a bounded linear operator on a Banach space $B$
with spectrum $\mathrm{spec}(L)=\Sigma \subset \mathbb C$.
If $\Sigma$ can be decomposed into two disjoint non-empty
sets $\Sigma_1$ and $\Sigma_2$ and if $\gamma$ is a simple closed
curve which is disjoint from $\Sigma$ and which has $\Sigma_1$ in its
interior and $\Sigma_2$ in its exterior then $\Pi : B \to B$ defined
by
\begin{align*}
\Pi = \frac{1}{2\pi i} \int_\gamma (z-L)^{-1} \, dz,
\end{align*}
is a projection (i.e. $\|\Pi\|=1$ and $\Pi^2=\Pi$). Moreover,
$B = B_1 \oplus B_2$, where $B_1=\Pi_1(B)$ and $B_2=(I-\Pi)(B)$ are
closed and $L$-invariant subspaces with $\mathrm{spec}(L|B_1) =
\Sigma_1$
and $\mathrm{spec}(L|B_2) =
\Sigma_2$.

Now condsider the operator $L_\psi : C^\alpha(\Lambda,\mathbb C) \to
C^\alpha(\Lambda,\mathbb C)$. By Theorem \ref{rpf}, we may decompose
its spectrum into $\Sigma_1= \{e^{P(\psi)}\}$ and a disjoint set 
$\Sigma_2$. Thus, we may decompose the operator 
$L_\psi$ as a sum
\begin{align*}
L_\psi^n = L_\psi^n \Pi + L_\psi^n (I-\Pi) 
= e^{nP(\psi)} \nu(\cdot)h + L_\psi^n(I-\Pi),
\end{align*}
where $\Pi_1 = \nu(\cdot) h$ is the projection onto the eigenspace
spanned by $e^{P(\psi)}$.  Furthermore, since $e^{P(\psi)}$ is
strictly maximal in modulus, we have
\begin{align*}
\lim_{n \to +\infty} \|L_\psi^n (I-\Pi)\|^{1/n} =
\sup\{|z| \hbox{ : } z \in \Sigma_2\} < e^{P(\psi)}.
\end{align*}
Choosing $\lambda_\psi$ slightly larger than
$\lim_{n \to +\infty} \|L_\psi^n (I-\Pi)\|^{1/n}$ completes the proof.
\end{proof}

\begin{corollary}
The quantities $e^{P(\psi)}$, $h$ and $\nu$ in Theorem \ref{rpf}
all depend analytically on $\psi$.
\end{corollary}

\begin{proof}
First we note that $L_\psi$ depends analytically on $\psi$. The result
is then a standard consequence of the fact that $e^{P(\psi})$ is an
isolated simple eigenvalue for $L_\psi$ \cite{BP, Kato}.
\end{proof}

Recall that we defined a function $-\phi$ to be {\it normalized} if,
for all $x \in \Lambda$,
\begin{align*}
\sum_{Ty=x} e^{-\phi(y)} =1.
\end{align*}
In particular, this condition implies that $-\phi$ is strictly negative.
We may rewrite this condition in terms of transfer operators as
$L_{-\phi} 1=1$. The following consequence of Theoreom \ref{rpf}
shows that, given a H\"older continuous function, we may find another
which is normalized and which had the same Gibbs measure.

\begin{corollary} \label{normalize}
Suppose that $\psi, h, \nu, \mu$ are as in Theorem \ref{rpf}.  Then 
\begin{align*}
-\phi := \psi + \log h - \log h \circ T -
  P(\psi) \in C^\alpha(\Lambda,\mathbb R)
\end{align*} 
is normalized, $L_{-\phi}^*\mu=\mu$ and $\mu$ is the Gibbs
  state 
$-\phi$. 
\end{corollary}

\begin{proof} Since $h>0$, $-\phi$ is well-defined. 
For any $m \in \mathcal M_T$,
\begin{align*}
h_T(m) + \int -\phi \, dm = h_T(m) + \int \psi  \, dm -P(\psi),
\end{align*}
so it follows that  $P(-\phi)=0$ and that $\mu$ is the Gibbs measure 
for $-\phi$.
We also have
\begin{align*}
L_{-\phi} 1(x) &= \sum_{Ty=x} e^{-\phi(y)} 
= \sum_{Ty=x} e^{\psi(y) +\log h(y) - \log h(Ty) - P(\psi)} \\
&= \frac{e^{-P(\psi)}}{h(x)} \sum_{Ty=x} e^{\psi(y)} h(y)
= \frac{e^{-P(\psi)}}{h(x)} L_\psi h(x) \\
&= \frac{e^{-P(\psi)}}{h(x)} e^{P(\psi)} h(x) =1,
\end{align*}
so $-\phi$ is normalized. By Theorem \ref{rpf}, $L_{-\phi}^*\mu=\mu$.
\end{proof}

To prove Theorem \ref{noncommmeas}, we shall need to consider a family
of transfer operators $L_{-t\phi}$, for $t \in \mathbb R$. By Theorem
\ref{rpf}, these will have a maximal eigenvalue equal to $e^{P(-t\phi)}$.
We end the section with a result on the 
regularity and derivative of the
function
$t \mapsto P(-t\phi)$.

\begin{lemma} \label{pressure}
The function $t \mapsto P(-t\phi)$ is real-analytic and strictly
decreasing.
Furthermore,
\begin{align*}
\left.\frac{dP(-t\phi)}{dt}\right|_{t=1} = -\int \phi \,
d\mu, 
\end{align*}
where $\mu$ is the Gibbs measure for $-\phi$.
\end{lemma}

\section{Proof of Theorem \ref{spectraltriple}}

In this section we proof that the $(H,A,D_\phi)$ we have constructed
is a spectral triple. The key point is that the locally constant
functions give a dense subalgebra of $C(\Lambda,\mathbb C)$ on which 
$\|[D_\phi,\pi(f)]\|$ is finite.

\begin{proof}[Proof of Theorem \ref{spectraltriple}]
Suppose that $f_1,f_2 \in C(\Lambda,\mathbb C)$ and that 
$\pi(f_1)=\pi(f_2)$. Then, in particular, by definition, for each $w \in W^*$,
$f_1(wy^{\mathfrak t(w)}) = f_2(wy^{\mathfrak t(w)})$.
The set $\{wy^{\mathfrak t(w)} \hbox{ : } w \in W^*\}$ is dense in
$\Sigma_A^+$ and
thus the set  $p(\{wy^{\mathfrak t(w)} \hbox{ : } w \in W^*\})$ is
dense in $\Lambda$. Hence $f_1 =f_2$ and 
$\pi : C(\Lambda,\mathbb C) \to B(H)$ is faithful.

It is clear from its definition that $D_\phi$ is
self-adjoint. The eigenvalues of $D_\phi$ are the numbers
$$\bigcup_{n=1}^\infty \{e^{\phi^n(wx^{\mathfrak t(w)})} \hbox{ : } w \in W_n\}$$
(counted with the appropriate multiplicity). In particular, $0$ is not 
an eigenvalue. Thus, the resolvent of $D_\phi$ is compact provided
$D_\phi^{-1}$ is compact and it is clear that $D_\phi^{-1}$, defined by
\begin{align*}
D_\phi^{-1}\left(\bigoplus_{w \in W^*} \left(\begin{matrix} \xi_1(w)
      \\ \xi_2(w) \end{matrix} \right)\right) 
= \bigoplus_{n=1}^\infty \bigoplus_{w \in W^*} e^{-\phi^n(wx^{\mathfrak t(w)})}
\left(\begin{matrix} 0 & 1 \\ 1 & 0 \end{matrix} \right)
\left(\begin{matrix} \xi_1(w) \\ \xi_2(w) \end{matrix} \right)
\end{align*}
is a compact operator. For $f \in C(\Lambda,\mathbb C)$,
\begin{align*}
&[D_\phi,\pi(f)]\left(\bigoplus_{w \in W^*} \left(\begin{matrix} \xi_1(w) \\
    \xi_2(w) \end{matrix} \right)\right) \\
&= \bigoplus_{w \in W^*} 
(f(wy^{\mathfrak t(w)})-f(wz^{\mathfrak t(w)}))
e^{\phi^n(wx^{\mathfrak t(w)})} 
\left(\begin{matrix} 0 & -1 \\ 1 & 0 \end{matrix}
\right)
\left(\begin{matrix} \xi_1(w) \\ \xi_2(w) \end{matrix} \right) \\
&= \bigoplus_{w \in W^*} 
(f(wy^{\mathfrak t(w)})-f(wz^{\mathfrak t(w)})) e^{\phi^n(wx^{\mathfrak t(w)})}
\left(\begin{matrix} -\xi_2(w) \\ \xi_1(w) \end{matrix}
\right). 
\end{align*}
Let $A_0 = LC(\Lambda)$, the subalgebra of locally constant functions 
on $\Lambda$. Recall that $A_0$ is dense in $A$. If $f \in A_0$ then there 
exists $N \geq 1$ such that
$$f(wy^{\mathfrak t(w)})=f(wz^{\mathfrak t(w)}) \quad 
\text{for all } w \in \bigcup_{n=N+1}^\infty
W_n.$$
Then
\begin{align*}
&\left\|[D,\pi(f)]\left(\bigoplus_{w \in W^*} \left(\begin{matrix} \xi_1(w) \\
    \xi_2(w) \end{matrix} \right)\right)\right\|_2^2 \\
&= \sum_{n=1}^\infty \sum_{w \in W_n}
(f(wy^{\mathfrak t(w)})-f(wz^{\mathfrak t(w)}))^2
e^{2\phi^n(wx^{\mathfrak t(w)})} ((-\xi_2(w))^2 + \xi_1(w)^2) \\
&\leq 2\|f\|_\infty e^{2N\|\phi\|_\infty}
\sum_{n=1}^N 
\sum_{w \in W_n}  ((-\xi_2(w))^2 + \xi_1(w)^2) \\
&\leq 2\|f\|_\infty e^{2N\|\phi\|_\infty} \|\xi\|_2^2. 
\end{align*}
Hence $\|[D,\pi(f)]\|<+\infty$. 
\end{proof}

\section{Proof of Theorem \ref{noncommmeas}}

 We will use the following version of the Hardy-Littlewood Tauberian
Theorem.
(See \cite{Co1}, Chapter IV, \S 2.$\beta$,
  Proposition 4.)

\begin{lemma} \label{hardylittlewood} 
Suppose that $A \in K(H)$ is a postive operator with eigenvalues 
$\{a_n\}_{n=1}^\infty$ (arranged in decreasing order)
and that $A \in \mathcal
L^{1,\infty}(H)$. 
Write
\begin{align*}
\zeta(t) = \sum_{n=1}^\infty a_n^t.
\end{align*}
Then
\begin{align*}
\lim_{t \to 1+}(t-1)\zeta(t) =L,
\end{align*}
if and only if
\begin{align*}
\lim_{n \to +\infty} \frac{1}{\log n} \sum_{k=1}^n a_k =L.
\end{align*}
\end{lemma}

We will suppose for the moment that $f \in C^\alpha(\Lambda,\mathbb R)$
and that $f \geq 0$,
so that $\pi(f)|D_\phi|^{-1}$ is a positive operator.
The eigenvalues of $\pi(f)|D_\phi|^{-1}$ are the numbers
$$\bigcup_{n=1}^\infty \{f(wy^{\mathfrak t(w)}) 
e^{-\phi^n(wx^{\mathfrak t(w)})}, f(wz^{\mathfrak t(w)})e^{-\phi^n(wx^{\mathfrak t(w)})}  
\hbox{ : } w \in
W_n\}$$
(counted with multiplicity). 
We define a spectral zeta function
\begin{align*}
\zeta_f(t) = \sum_{n=1}^\infty \sum_{w \in W_n} \left(\left(f(wy^{\mathfrak t(w)})
  e^{-\phi^n(wx^{\mathfrak t(w)})}\right)^t
+\left(f(wz^{\mathfrak t(w)}) e^{-\phi^n(wx^{\mathfrak t(w)})}\right)^t\right)
\end{align*}
and we also write
\begin{align*}
&\zeta_{f,x}(t) = \sum_{n=1}^\infty \sum_{w \in W_n} \left(f(wx^{\mathfrak t(w)})
  e^{-\phi^n(wx^{\mathfrak t(w)})}\right)^t, \\
&\zeta_{f,y}(t) =\sum_{n=1}^\infty \sum_{w \in W_n}
\left(f(wy^{\mathfrak t(w)}) e^{-\phi^n(wx^{\mathfrak t(w)})}\right)^t,
\ \mathrm{and} \\
&\zeta_{f,z}(t) =\sum_{n=1}^\infty \sum_{w \in W_n}
\left(f(wz^{\mathfrak t(w)}) e^{-\phi^n(wx^{\mathfrak t(w)})}\right)^t.
\end{align*}

In order to study these functions, it will be convenient to introduce
another
one which is easier to express in terms of transfer operators. Hence
we define
\begin{align*}
\eta_f(s) = \sum_{n=1}^\infty \sum_{w \in W_n} f(wx^{\mathfrak t(w)}) e^{-s\phi^n(wx^{\mathfrak t(w)})}.
\end{align*}

\begin{lemma} \label{residue} For $f \in C^\alpha(\Lambda,\mathbb R)$
  with
$f >0$,
$\eta_f(t)$ converges 
for
$t>1$ and 
\begin{align*}
\lim_{t \to 1+}(t-1)\eta_f(t)= \left(\frac{\int f \, d\mu}
{\int \phi \, d\mu}\right) 
\left(\sum_{j=1}^k
(L_{-\phi}\chi_j)(x^{(j)})\right).
\end{align*}
\end{lemma}

\begin{proof}
Provided $\eta_f(t)$ converges, we may use the definitions of
$x^{(j)}$ and
$L_{-t\phi}$ to write
\begin{align*}
\eta_f(t) &= \sum_{n=1}^\infty \sum_{w \in W_n}
f(wx^{\mathfrak t(w)}) e^{-t\phi^n(wx^{\mathfrak t(w)})} 
\\
&= \sum_{n=1}^\infty \sum_{j=1}^k L_{-t\phi}^n(\chi_j \circ T^{n-1} f)(x^{j}) 
\\
&= \sum_{n=1}^\infty \sum_{j=1}^k 
L_{-t\phi} (\chi_j L_{-t\phi}^{n-1}f)(x^{j}).
\end{align*}

By Theorem \ref{rpf} and Lemma \ref{pressure}, for $t >1$, 
$L_{-t\phi}$ has spectral radius $e^{P(-t\phi)}<1$.  Thus, 
using the spectral radius formula, 
it is easy to see that $\eta_f(t)$ converges. Furthermore, 
by Corollary \ref{spectraldecomp}, we have
\begin{align*}
\eta_f(t)
&= \sum_{n=0}^\infty \sum_{j=1}^k \left(\int f \, d\nu_t\right) e^{nP(-t\phi)}
\left(L_{-t\phi}(\chi_j h_t)\right)(x^{j}) +\sum_{n=0}^\infty q_n(t) \\
&= \left(\int f \, d\nu_t\right) \sum_{n=0}^\infty e^{nP(-t\phi)}
\left(\sum_{j=1}^k \left(L_{-t\phi}(\chi_j h_t)\right)(x^{j}) \right)
+\sum_{n=1}^\infty q_n(t)
\\
&=\left(\int f \, d\nu_t\right)
\frac{\sum_{j=1}^k \left(L_{-t\phi}(\chi_j h_t)\right)(x^{j}) }{1-e^{P(-t\phi)}}
+ 
\sum_{n=1}^\infty q_n(t), 
\end{align*}
where $h_t$ and $\nu_t$ are the eigenfunction and eigenmeasure
for $L_{-t\phi}$ given by Theorem \ref{rpf} and where
$q_n(t) = O(\lambda_{-t\phi}^n)$ (with $\lambda_{-t\phi} < e^{P(-t\phi)}$).
Since 
\begin{enumerate}
\item[(i)]
$t \mapsto e^{P(-t\phi)}$, $t \mapsto h_t$ and $t \mapsto \nu_t$
are all analytic;
\item[(ii)]
$e^{P(-\phi)}=1$, $h_1=1$ and $\nu_1=\mu$; and
\item[(iii)]
\begin{align*}
\left.\frac{de^{P(-t\phi)}}{dt}\right|_{t=1} = -\int \phi  \, d\mu;
\end{align*}
\end{enumerate}
we see that
$$\eta_f(t)= \left(\frac{\int f \, d\mu}{\int \phi \, d\mu}\right) 
\left(\sum_{j=1}^k
(L_{-\phi}\chi_j)(x^{j})\right)
\frac{1}{t-1}
+ a(t),$$
where $a(t)$ is finite for $t \geq 1$. 
\end{proof}

\begin{remark}
In fact, one can show (using the type of methods described in 
\cite{PP, Sh1}) that, considered as a function of a complex variable
$s$, $\eta_f(s)$ is analytic for $\mathrm{Re}(s)>1$, has a simple pole
at $s=1$ and, provided the sums of $\phi$ around periodic orbits do
not
all lie in a discrete subgroup of $\mathbb R$, apart from this pole,
$\eta_f(s)$ has an analytic
extension to a neighbourhood of $\mathrm{Re}(s) \geq 1$.
\end{remark}

\begin{lemma} \label{relatezetaeta}
For $f \in C^\alpha(\Lambda,\mathbb R)$
  with
$f >0$, $\zeta_{f,y}(t)$ and $\zeta_{f,z}(t)$ converge for $t>1$.
Furthermore,
we have
\begin{align*}
\lim_{t \to 1+} (t-1)\zeta_{f,y}(t) 
= \lim_{t \to 1+} (t-1)\zeta_{f,z}(t)
= \lim_{t \to 1+} (t-1)\eta_f(t).
\end{align*}
\end{lemma}

\begin{proof}
First we shall show that 
it suffices to consider $\zeta_{f,x}(t)$.
Note that, for $t>1$,
\begin{align*}
|\zeta_{f,x}(t) - \zeta_{f,y}(t)| &\leq \sum_{n=1}^\infty \sum_{w \in W_n} 
\left|f(wx^{\mathfrak t(w)})^t
- f(wy^{\mathfrak t(w)})^t \right| e^{-t\phi^n(wx^{\mathfrak t(w)})} \\
&\leq t \|f\|_\infty^{t-1} \sum_{n=1}^\infty \sum_{w \in W_n} 
|f(wx^{\mathfrak t(w)})-f(wy^{\mathfrak t(w)})| e^{-t\phi^n(wx^{\mathfrak t(w)})}.
\end{align*}

Let $\{w_m\}_{m=1}^\infty$ be any enumeration of $W^*$.
Then 
\begin{align*}
\lim_{m \to +\infty} d((w_mx^{\mathfrak t(w)},w_my^{\mathfrak
  t(w)})=0,
\end{align*}
so that,
since $f$ is continuous,
\begin{align*}\lim_{m \to +\infty} f(w_mx^{\mathfrak t(w)}) 
-f(w_my^{\mathfrak t(w)})=0.
\end{align*}
Thus, since each set $W_n$ is finite, given $\epsilon >0$, 
there exists $N \geq 1$ such that if $w \in W_n$ and $n \geq N$ then
$|f(wx^{\mathfrak t(w)}) -f(wy^{\mathfrak t(w)})|<\epsilon$. Thus,
\begin{align*}
\left|\zeta_{f,x}(t) - \zeta_{f,y}(t)\right| &\leq 2t\|f\|_\infty^t (N-1)
+ \epsilon t \|f\|_\infty^{t-1}
\left(\sum_{n=N}^\infty \sum_{w \in W_n}e^{-t\phi^{n}(wx^{\mathfrak t(w)})}\right) \\
&\leq 2t\|f\|_\infty^t (N-1) + \epsilon t \|f\|_\infty^{t-1} \eta_1(t).
\end{align*}
Hence, $\zeta_{f,y}(t)$ converges provided $\zeta_{f,x}(t)$ converges
and we have the estimate
\begin{align*}\lim_{t \to 1+} (t-1)(\zeta_{f,x}(t)-\zeta_{f,y}(t))
&\leq \epsilon \lim_{t \to 1+} (t-1) \eta_1(t) \\
&= \epsilon \left(\frac{1}{\int  \phi \, d\mu}\right) 
\left(\sum_{j=1}^k
(L_{-\phi}\chi_j)(x^{j})\right).
\end{align*}
Since $\epsilon >0$ is arbitrary, this shows that
$\lim_{t \to 1+} (t-1)\zeta_{f,x}(t) 
= \lim_{t \to 1+} (t-1)\zeta_{f,y}(t)$.
A similar argument for $\zeta_{f,z}(t)$ completes the proof of the claim.

To complete the proof, we notice that, as $t \to 1+$, we have
\begin{align*}
f(wx^{\mathfrak t(w)})^t - f(wx^{\mathfrak t(w)}) &= f(wx^{\mathfrak
  t(w)})
\left(f(wx^{\mathfrak t(w)})^{t-1} -1\right) \\
&= f(wx^{\mathfrak t(w)}) (t-1 + O((t-1)^2)).
\end{align*}
Thus
\begin{align*}
\zeta_{f,x}(t) - \eta_f(t)
= (t-1 + O((t-1)^2)) \eta_f(t),
\end{align*}
so that $\zeta_{f,x}(t)$ converges for $t>1$ and 
\begin{align*}
\lim_{t \to 1+} (t-1)\zeta_{f,x}(t) = \lim_{t \to 1+} (t-1)\eta_f(t),
\end{align*}
as required.
\end{proof}

\begin{proof}[Proof of Theorem \ref{noncommmeas}]
We need to show that, whenever $f \in C(\Lambda,\mathbb R)$ with $f
\geq 0$, we have
\begin{align*}
\lim_{n \to +\infty} \frac{1}{\log n} \sum_{k=1}^n a_k(f) = c_\phi
\int f\,
d\mu,
\end{align*}
where $\{a_k(f)\}_{k=1}^\infty$ are the eigenvalues of $\pi(f)
|D_\phi|^{-1}$,
counted with multiplicity and written in decreasing order,
and where
\begin{align*}
c_\phi = 
\frac{2}{\int \phi \, d\mu} \sum_{j=1}^k \sum_{Tx =x^j}
e^{-\phi(x)}
\chi_j(x)= \frac{2}{\int \phi \, d\mu}\sum_{j=1}^k (L_{-\phi} \chi_j)(x^j).
\end{align*}

First, suppose that  $f \in C^\alpha(\Lambda,\mathbb R)$ and that
$f \geq 0$.
Lemma \ref{residue} and Lemma \ref{relatezetaeta} show that
$\zeta_f(t)$ converges for $t>1$ and diverges for $t=1$.
Thus $\pi(f) |D|^{-1} \in \mathcal L^{1,\infty}(H)$.
It follows immediately from Lemma \ref{hardylittlewood} 
and Lemma \ref{relatezetaeta} that 
\begin{align*} 
\lim_{n \to +\infty} \frac{1}{\log n} \sum_{k=1}^n a_k(f)
=\lim_{t \to 1+} (t-1) \zeta_f(t)
= c_\phi \int f \, d\mu.
\end{align*}

Now suppose $f \in C(\Lambda,\mathbb R)$ and $f \geq 0$. Given
$\epsilon>0$,
 we may choose
$g_1,g_2 \in  C^\alpha(\Lambda,\mathbb C)$ such that 
$
0 \leq g_1 \leq f \leq g_2$ and
\begin{align*}
\int f \, d\mu - \epsilon \leq \int g_1 \, d\mu \leq
\int g_2 \, d\mu \leq \int f \, d\mu + \epsilon.
\end{align*}
Then
we have
\begin{align*}
c_\phi \left(\int f \, d\mu - \epsilon\right) 
&\leq c_\phi \int g_1 \, d\mu 
= \lim_{t \to 1+} (t-1)\zeta_{g_1}(t) \\
&\leq \liminf_{t \to 1+} (t-1) \zeta_f(t) 
\leq \limsup_{t \to 1+} (t-1) \zeta_f(t) \\
&\leq \lim_{t \to 1+} (t-1)\zeta_{g_2}(t) 
= c_\phi \int g_2 \, d\mu \\
&\leq c_\phi \left(\int f \, d\mu - \epsilon\right).
\end{align*}
Since $\epsilon>0$ is arbitrary, the required convergence
result holds for $f$.
\end{proof}

\end{document}